\documentclass[12pt]{article}

\usepackage{graphicx, amssymb, latexsym, amsfonts, amsmath, lscape, amscd,
amsthm, color, epsfig, mathrsfs, tikz, enumerate}

\newtheorem{theorem}{Theorem}[section]

\newtheorem{example}[theorem]{Example}
\newtheorem{lemma}[theorem]{Lemma}
\newtheorem{proposition}[theorem]{Proposition}

\newtheorem{observation}[theorem]{Observation}

\newcommand\DELETE[1]{}

\begin{document}

%\begin{frontmatter}

\title{{\bf On fractional version of oriented coloring\footnote{This work is partially supported by the IFCAM project ``Applications of graph homomorphisms'' (MA/IFCAM/18/39).}}}
\author{
{\sc Sandip Das}$\,^{a}$, {\sc Soham Das}$\,^{a}$, {\sc Swathy Prabhu}$\,^{b}$, {\sc Sagnik Sen}$\,^{c}$ \\
\mbox{}\\
{\small $(a)$ Indian Statistical Institute, Kolkata, India}\\
{\small $(b)$ Ramakrishna Mission Vivekananda Educational and Research Institute, Kolkata, India}\\
{\small $(c)$ Indian Institute of Technology  Dharwad, Dharwad, India}
}

\date{\today}

\maketitle

\begin{abstract}
We introduce  the fractional version of oriented coloring and initiate its study. We prove some basic results and study the parameter for directed cycles and sparse planar graphs. In particular, we show that 
for every $\epsilon > 0$, there exists an integer $g_{\epsilon} \geq 12$ 
such that any oriented planar graph having girth at least $g_{\epsilon}$ has 
fractional oriented chromatic number at most $4+\epsilon$. 
Whereas, it is known that there exists 
an oriented planar graph having girth at least $g_{\epsilon}$ with oriented chromatic number equal to $5$. We also study the fractional oriented chromatic number
of  directed cycles and provide its exact value. 
Interestingly, the result 
depends on the prime divisors of the length of the directed cycle. 
\end{abstract}


\section{Introduction}
``It is possible to go to a graph theory conference and to ask oneself, at the end of every talk,
what is the fractional analogue?'' - Scheinermann and Ullman~\cite{scheinermann2008fgt} made this remark in the preface of their book on fractional graph theory. Considering the popularity of studying fractional versions of well-studied problems, it is a wonder that the fractional version of oriented coloring is yet to be studied. We initiate it with this article.

An \textit{oriented graph} $G$ is a directed graph without any directed 
cycle of length 1 or 2.
In this article, a graph $G$ refers to a simple or an oriented graph while $V(G)$ denotes the set of vertices of $G$ and $E(G)$ or $A(G)$ refers to its set of 
edges or arcs, respectively.

The notion of oriented coloring was introduced by Bruno 
Courcelle~\cite{courcelle1994monadic} in 1994
inspiring a considerable volume of work on 
the topic (see recent survey~\cite{sopena-updated-survey} for details). 
An \textit{oriented $k$-coloring} of an oriented graph $G$ is a function $f$ from $V(G)$ to a set of $k$ colors such that 
(i) $f(x) \neq f(y)$ for every arc $xy \in A(G)$, and 
(ii) $f(y) = f(z)$ implies $f(x) \neq f(w)$ for every $xy,zw \in A(G)$.
The \textit{oriented chromatic number} $\chi_o(G)$ is the minimium $k$ such that 
$G$ admits an oriented $k$-coloring. 
The oriented chromatic number for a family 
$\mathcal{F}$ of oriented graphs is given by 
$$\chi_o(\mathcal{F}) = \max\{\chi_o(G)| G \in \mathcal{F}\}.$$

Without further ado, let us now define the natural analogue of the fractional version of 
the oriented chromatic number.
Let $S$ be a set of $k$ colors and let $P_b(S)$ denote the set of all subsets of $S$ having cardinality $b$. 
A \textit{$b$-fold oriented $k$-coloring} is a   
mapping $f$ from $V(G)$ to $P_b(S)$ satisfying 
$(i)$ $f(x) \cap f(y) = \emptyset $ for every arc  $xy \in A(G)$, and 
$(ii)$  $f(x) \cap f(w) \neq \emptyset $ implies 
$f(y) \cap f(z) = \emptyset $ for every $xy, zw \in A(G)$.   
The \textit{$b$-fold oriented chromatic number} $\chi^b_o(G)$  of $G$ is the minimum $k$
such that $G$ admits an $b$-fold oriented $k$-coloring.

The \textit{fractional oriented chromatic number} of $G$ is given by
$$\chi^*_o(G) = \inf_{b \rightarrow \infty} 
\frac{\chi^b_{o}(G)}{b}.$$
Observe that $\chi^{a+b}_o(G) \leq \chi^a_o(G) + \chi^b_o(G)$ for all $a,b \in \mathbb{N}$. Thus the above limit exists due to 
the 
Subadditivity Lemma (see Appendix A.4~\cite{scheinermann2008fgt}).
Naturally, for a  family 
$\mathcal{F}$ of oriented graphs 
$$\chi^*_o(\mathcal{F}) = \max\{\chi^*_o(G)| G \in \mathcal{F}\}.$$
Notice that  the oriented $k$-coloring and chromatic number $\chi_o(\cdot)$  are equivalent to the  $1$-fold oriented $k$-coloring and chromatic number 
$\chi^1_o(\cdot)$, respectively.

The close relation between oriented chromatic number and 
(oriented) graph homomorphisms is well-known~\cite{sopena-updated-survey}. On the other hand, the fractional version of the usual chromatic number has a famous equivalent formulation using homomorphism to Kneser graphs~\cite{scheinermann2008fgt}. Thus, naturally we study the relation between the oriented fractional coloring and oriented graph homomorphisms. Alongside such a study, we 
explore some other basic properties of oriented fractional coloring.

A relevant related concept is the oriented analogues of clique and clique number. 
The analogue of clique and clique number   for oriented graphs ramified into two notions: 
(i) oriented absolute clique and clique number, and (ii) oriented relative clique and clique number. The later is more relevant here. 
An \textit{oriented relative clique}~\cite{nandySS2016} of an oriented graph $G$
is a vertex subset  $R \subseteq V(G)$ satisfying $f(x) \neq f(y)$ under any homomorphism $f$ of $G$, and for any distinct vertices $x,y \in R$. 
The  \textit{oriented relative clique number}~\cite{nandySS2016} $\omega_{ro}(G)$ of $G$ is 
the cardinality of a largest oriented relative clique. The oriented relative clique number 
$\omega_{ro}(\mathcal{F})$ for a family $\mathcal{F}$ of oriented graphs is the maximum 
$\omega_{ro}(G)$ where $g \in \mathcal{F}$. 

We will soon see that the parameter $\chi^*_o(\cdot)$ is sandwiched between the parameters oriented relative clique number (lower bound) and 
oriented chromatic number (upper bound). 
Therefore, it is interesting to examine the oriented fractional chromatic number of oriented graphs or 
families of oriented graphs for which the values of $\omega_{ro}$ and $\chi_o$ are different. 

The most ordinary such graphs one can think of are probably the directed cycles $C_r$ of length $r$. 
The study of fractional oriented chromatic number of these simple looking graphs turn out to be quite challenging. Moreover,  the corresponding results are utterly surprising and have 
an interesting relation with prime numbers. 
To motivate the readers, we present an interesting example where the fractional oriented chromatic number is strictly less than the oriented chromatic number. 

\begin{example}\label{example 7-cycle}
Let $C_7$ be the directed 7-cycle. We know that $\chi_o(C_7) =  4$. 
However,  Fig.~\ref{fig ex1} shows a 
$2$-fold oriented $7$-coloring of $C_7$ implying
$\chi^*_o(C_7) \leq \frac{7}{2} = 3.5 < 4 = \chi_o(C_7)$.  Is this upper bound tight? We will answer that later in this article, and till then we encourage the readers to think about it. 
\end{example}

\begin{figure}
\centering
	\begin{tikzpicture}[inner sep=.7mm]
	
	\foreach \a in {0,...,6} 
	{
		\node[draw, circle, line width=1pt](v\a) at (\a*360/7:3cm){$x_\a$};	
	}

	\node at (0:3.3cm)[right]{$\{1,2\}$};	
	\node at (360/7:3.3cm)[above]{$\{3,4\}$};	
	\node at (2*360/7:3.4cm)[left]{$\{5,6\}$};	
	\node at (3*360/7:3.4cm)[left]{$\{7,1\}$};
	\node at (4*360/7:3.3cm)[left]{$\{2,3\}$};
	\node at (5*360/7:3.3cm)[below]{$\{4,5\}$};
	\node at (6*360/7:3.35cm)[right]{$\{6,7\}$};

	\draw[-latex,line width=1pt,black] (v0) -- (v1);
	\draw[-latex,line width=1pt,black] (v1) -- (v2);
	\draw[-latex,line width=1pt,black] (v2) -- (v3);
	\draw[-latex,line width=1pt,black] (v3) -- (v4);
	\draw[-latex,line width=1pt,black] (v4) -- (v5);
	\draw[-latex,line width=1pt,black] (v5) -- (v6);
	\draw[-latex,line width=1pt,black] (v6) -- (v0);

 	\end{tikzpicture}
 	\caption{Fractional  oriented  coloring of the directed 7-cycle $C_7$.}
 	\label{fig ex1}
 \end{figure}
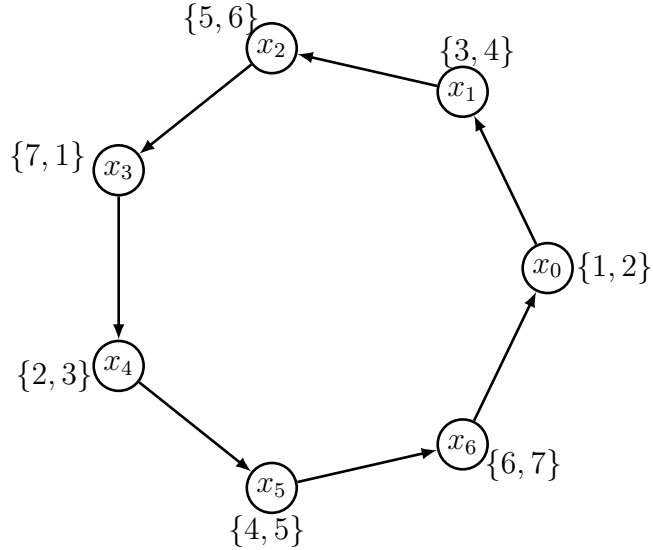

One of the most important open problems in the domain of oriented coloring is determining the oriented chromatic number of the family $\mathcal{P}_3$ of oriented planar graphs. In particular, improving the upper bound  
$\chi_o(\mathcal{P}_3) \leq 80$~\cite{raspaud1994planar80} seems to be especially challenging.  
Moreover, the related class of questions of determining the oriented chormatic number 
$\chi_o(\mathcal{P}_g)$  of the family $\mathcal{P}_g$ of oriented planar graphs with \textit{girth} 
(length of a smallest cycle) at least $g$  are also interesting.   
However, till date 
the only known exact bounds are for the cases 
$\chi_o(\mathcal{P}_g) = 5$ for all $g \geq 12$.

In general, finding the exact value of $\chi_o(\mathcal{P}_g)$ for $g \geq 3$ seems to be tough problem. Therefore, we wondered if the fractional version will be any easier or if the exact values of the parameter will remain the same. For this particular article, we focus more on the second question and find that indeed the fractional oriented chromatic number of $\mathcal{P}_g$ can be less than $5$ for large values of $g$, and, in fact, it gets arbitrarily close to $4$ as $g \to \infty$.

The organization of this article as follows. In Section~\ref{sec pre}, we present the preliminary notation, terminology and establish some basic results. In Sections~\ref{sec directed cycle} and~\ref{sec planar} we study the fractional oriented chromatic number of directed graphs and sparse planar graphs, respectively. Finally, in Section~\ref{sec conclusions}, we share our concluding remarks. 

\section{Preliminaries and basic results}\label{sec pre}
Let $G$ and $H$ be two oriented  graphs. 
A function 
$f: V(G) \to V(H)$ is a \textit{homomorphism} of $G$ to $H$ 
if for each arc $xy$ of $G$, $f(x)f(y)$ 
is an arc of $H$. We use the notation $G \to H$ to denote that $G$ admits a homomorphism to $H$. 
This definition readily motivates the following result. 

\begin{proposition}
If $G \to H$, then $\chi_o^{*}(G)  \leq \chi_o^{*}(H)$. 
\end{proposition}

\begin{proof}
Let $\chi_o^{*}(H) = \frac{\chi^b_{o}(H)}{b}$, and thus there exists a $b$-fold oriented $\chi^b_{o}(H)$-coloring $f$ of $H$. Assume that $\phi: G \to H$ is a homomorphism. Notice  that the composition 
$f \circ \phi$ is a 
$b$-fold oriented $\chi^b_{o}(H)$-coloring of $G$.
\end{proof}

Now the next natural question is whether there exists any equivalent definition of 
oriented fractional coloring and chromatic number using the notion of homomorphisms. 
To express such an equivalent formulation we need some definitions.

The \textit{Kneser graph $KG_{a,b}$} is a graph having the subsets of cardinality $b$ of a set of cardinality $a$ as vertices, and two vertices are adjacent if and only if they are disjoint sets. 
A \textit{consistent sub-orientation} of $KG_{a,b}$ is an oriented graph $\overrightarrow{KG}_{a,b}$, whose underlying graph is a subgraph of $KG_{a,b}$, and whose arcs are oriented in such a way that
given two arcs $xy$ and $wz$ we have $x \cap z \neq \emptyset \implies y \cap w = \emptyset$. 
That brings us to the equivalent formulation of oriented fractional coloring using the notion of homomorphisms.

\begin{theorem}\label{th kneser gen}
An oriented graph $G$ satisfies $\chi_o^{*}(G) \leq \frac{a}{b}$ if and only if 
$G \to \overrightarrow{KG}_{ac,bc}$ where $\overrightarrow{KG}_{ac,bc}$ is a consistent sub-orientation of 
$KG_{ac,bc}$ and $c$ is a constant positive integer. 
\end{theorem}

To prove the above theorem, we will first prove a supporting result.

\begin{theorem}\label{th kneser sp}
If $G \to \overrightarrow{KG}_{a,b}$ for some consistent sub-orientation of 
$KG_{a,b}$, then there exists a consistent sub-orientation of 
$KG_{ac,bc}$ satisfying $G \to \overrightarrow{KG}_{ac,bc}$ for every positive integer $c$. 
\end{theorem}

\begin{proof}
It is enough to show that there exists a consistent sub-orientation of 
$KG_{ac,bc}$ isomorphic to $\overrightarrow{KG}_{a,b}$. This can be obtained by 
replacing each vertex $x = \{x_1, x_2, \cdots, x_b\}$ of $\overrightarrow{KG}_{a,b}$ with 
$\hat{x} = \{x_{11}, x_{12}, \cdots, x_{1b}, x_{21}, x_{22}, \cdots, x_{2b},\cdots,  x_{c1}, x_{c2}, \cdots, x_{cb}\}$.
\end{proof}

Now we are ready to prove Theorem~\ref{th kneser gen}. 

\medskip

\noindent \textit{Proof of Theorem~\ref{th kneser gen}.}
Let $\chi_o^{*}(G) \leq \frac{a}{b}$. That means, $G$ admits a $b'$-fold oriented $a'$-coloring $f$ for some $a', b'$ satisfying $\frac{a'}{b'} = \frac{a}{b}$. Notice that, as we have not assumed a reduced form of $\frac{a}{b}$, it is possible to have $a' < a$ and $b' < b$.  However, due to Theorem~\ref{th kneser sp} it is enough to prove assuming a reduced form of $\frac{a}{b}$.

Now let us construct an oriented  graph $C$ having the colors used for the 
$b'$-fold oriented  $a'$-coloring $f$ of $G$  as vertices. 
Moreover, two vertices $x$ and $y$ of $C$ have an arc $xy$ between them if there exists an arc $uv$ in $G$ satisfying $f(u)=x$ and $f(v) = y$. Notice that, $C$ is a 
consistent sub-orientation of 
$KG_{a',b'}$ and $f$ is a homomorphism of $G$ to $C$. This takes care of the ``if'' part. 

On the other hand, if $g: G \to \overrightarrow{KG}_{ac,bc}$ where 
$\overrightarrow{KG}_{ac,bc}$ is a consistent sub-orientation of 
$KG_{ac,bc}$, then $f$ itself is also a $bc$-fold $ac$-coloring of $G$. This implies $\chi_o^{*}(G) \leq \frac{a}{b}$. This completes the ``only if'' part. 
\qed

Next we will move onto the sandwich theorem mentioned in the introduction. Before stating it, it is useful to recall 
a handy characterization of an oriented relative clique.

\begin{proposition}~\cite{nandySS2016}\label{prop oclique char}
A vertex subset $R \subseteq V(G)$ of $G$ is an  oriented relative clique 
if and only if 
any non-adjacent pair of  
vertices in $R$ is connected by a directed 2-path.  
\end{proposition}

Now we are ready to state and prove the sandwich theorem. 

\begin{theorem}\label{prop sandwich}
For any  oriented graph $G$, 
$\omega_{ro}(G) \leq \chi^*_o(G) \leq \chi_o(G)$. 
\end{theorem}

\begin{proof}
Let $R$ be a relative oriented clique of $G$ with $|R| = \omega_{ro}(G)$. 
Suppose $f$ is a $b$-fold oriented $a$-coloring of $G$. Then, as the non-adjacent vertices of $R$ are connected by a directed $2$-path, we must have $f(x) \cap f(y) = \emptyset$ for all distinct $x,y \in R$. Thus, a total of $|R|b$ colors have been used on the vertices of $R$. 
Therefore, $a \geq |R|b = \omega_{ro}(G)b$ which implies that $\frac{a}{b} \geq \omega_{ro}$. Hence the first inequality.

The second inequality follows from the trivial observation 
$\chi^{*}_o(G) \leq \chi^{1}_o(G) = \chi_o(G)$. 
\end{proof}

We will sign off the section by establishing a general lower bound for oriented fractional chromatic number. To do so, we will introduce an oriented analogue of independent sets. 
A vertex subset $I \subseteq V(G)$ of an oriented graph $G$ is an \textit{oriented independent set} if for each distinct pair of vertices $x, y \in I$, it is possible to find an oriented coloring $f$ of $G$ satisfying $f(x)=f(y)$.

\begin{proposition}
Given an oriented graph $G$, a vertex subset $I \subseteq V(G)$ is an oriented independent set if and only if any two vertices of $I$ are neither adjacent nor connected by a directed $2$-path. 
\end{proposition}

\begin{proof}
The ``if'' part directly follows  from Proposition~\ref{prop oclique char}.  

For the ``only if'' part, given distinct $x, y \in I$, define the following function
$$f(z) = \begin{cases}
z, &\text{ for } z \in V(G) \setminus \{y\},\\
x, &\text{ for } z =y.
\end{cases}
$$
Notice that $f$ is a homomorphism of $G$ to $G \setminus \{y\}$. 
\end{proof}

We also use this definition to introduce the oriented analogue of the parameter independence number. 
The \textit{oriented independence number $\alpha_o(G)$} is the maximum $|I|$ where $|I|$ is an oriented independent set of $G$. 
Finally, we are ready to state the result that gives us a general lower bound of oriented fractional chromatic number. 

\begin{theorem}
Given any oriented graph $G$ we have $\chi_o^{*}(G) \geq \frac{|V(G)|}{\alpha_o(G)}$. 
\end{theorem}

\begin{proof}
Let $f$ be a $b$-fold $a$-coloring 
of $G$ with $\frac{a}{b} = \chi_o^{*}(G)$. 
This means, there is a set $S$ of $a$ colors whose set of all subsets having cardinality $b$ is $P_b(S)$, and in particular $f$ is a function from $V(G)$ to $P_b(S)$. 

Note that, if $f(u)=x$ and $f(v)=y$, and $x \cap y \neq emptyset$, 
then $u, v$ are independent. 
Thus $S_z = \{u: z \in f(u)\}$ for any $z \in S$ is an oriented independent set and 
$|S_z| \leq \alpha_o(G)$. 
Moreover, observe that every vertex $u \in V(G)$ is 
part of exactly $b$ such sets. 
 Thus we have
$$|V(G)| \cdot b = \sum_{z \in S} |S_z| \leq \alpha_o(G) \cdot |S|  =  \alpha_o(G) \cdot a.$$
This implies 
$$\chi_o^{*}(G) = \frac{a}{b} \geq \frac{|V(G)|}{\alpha_o(G)}$$
 and completes the proof. 
\end{proof}

Notice that, we can prove the tightness of the bound $\chi_o^{*}(C_7) = 3.5$ easily using the above result. In the next section, we are going to provide the generalized tight values for 
$\chi_o^{*}(C_r)$, where $C_r$ denotes the directed cycle of length $r$, for all $r \geq 3$.

\section{Directed cycles}\label{sec directed cycle}
We use the following classification of prime numbers to present our result.
A prime number $k > 3$ is a \textit{type-A prime} if 
$k\equiv 3\bmod 4$, and 
 is a \textit{type-B prime} if 
  $k\equiv 1\bmod 4$. 
  Let $r > 5$ be a positive integer. 
  If $r$ does not have a type-A prime factor, then 
   $\beta(r)=0$. Otherwise, $\beta(r)=\frac{4}{p+1}$ where $p$ is 
   the least type-A prime factor of $r$. 
 For instance 
$\beta(28)=\beta(35)=\frac{1}{2}$, 
$\beta(55)=\beta(88)=\frac{1}{3}$,
 and $\beta(26) =\beta(2^n)=0$.

\begin{theorem}\label{directed-cycles}
Let $C_r$ be a directed cycle of length $r$. Then

\begin{enumerate}[(a)]
\item $\chi^*_{o}(C_r) = 3$ if  $r \equiv 0 \bmod 3$,
\item $\chi^*_{o}(C_4) = 4$,
\item $\chi^*_{o}(C_5) = 5$,
\item $\chi^*_{o}(C_r) = 4 - \beta(r)$, if $r > 5$ and   
$r \not\equiv 0 \bmod 3$.
\end{enumerate}
\end{theorem}

The first three parts of the above theorem follow directly from known results.  However, the main challenge is to prove the last part of it. 
That part is unexpected and gives us another indication of how difficult, counter-intuitive and yet beautiful the theory of 
fractional oriented coloring 
may actually be.

Throughout this section we will assume that the set of vertices of 
$C_r$ is $V(C_r) = \{u_i | i \in \mathbb{Z}/r\mathbb{Z} \}$ and the set of arcs of $C_r$ is  $$A(C_r) = \{u_iu_{i+1} | \text{  the operation `+' is taken modulo } r \}.$$

\medskip

\noindent  \textit{Proof of Theorem~\ref{directed-cycles}(a,b,c).}
We know that $\omega_{ro}(C_r) = \chi_o(C_r) = 3$ 
for all $r \equiv 0 \bmod 3$~\cite{sopena-updated-survey}. We also know that 
$\omega_{ro}(C_4) = \chi_o(C_4) = 4$ and 
$\omega_{ro}(C_5) = \chi_o(C_5) = 5$~\cite{sopena-updated-survey}. Thus the result follows due to Proposition~\ref{prop sandwich}. \hfill $ \square $

\medskip

The proof of Theorem~\ref{directed-cycles}(d) is broken into several 
lemmas and 
observations presented in the following.

For the rest of the section we  only deal with $C_r$ having $r > 5$ and 
$r \not\equiv 0 \bmod 3$. All the `$+$' and `$-$' operations perfomed in 
 the subscript of $u$ is computed modulo $r$ unless otherwise stated.

\begin{lemma}\label{greedy-coloring}
For all $r > 5$ and $r \not\equiv 0 \bmod 3$ we have, 
$3 \leq \chi^*_{o}(C_r) \leq 4$.
\end{lemma}

\begin{proof}
We know that 
$\omega_{ao}(C_r) = 3$ and $\chi_o(C_r) = 4$~\cite{sopena-updated-survey}.  Thus the result follows due to Proposition~\ref{prop sandwich}.
\end{proof}

Any $b$-fold oriented $k$-coloring 
of $C_r$ having 
$\frac{k}{b} < 4$ is a \textit{miser $b$-fold oriented $k$-coloring}. 

\medskip

Let $c$ be a miser $b$-fold oriented $k$-coloring of $C_r$ using a set $S$ of $k < 4b$ colors. Note that as 
a directed 2-path is an oriented clique, its vertices must receive 
disjoint sets of colors. As a consequence:

\begin{lemma}\label{lem dipath}
For all $i,j \in \{0,1, \cdots , r-1\}$  satisfying 
$0 \neq |i-j| \leq 2$ we have $c(u_i) \cap c(u_j) = \emptyset$.  
\end{lemma}

Thus assume that $c(u_0) = A$, $c(u_1) = B$ and $c(u_2) = C$. By Lemma~\ref{lem dipath}, $|A \cup B \cup C| = 3b$ and
$|A|=|B|=|C| = b$. 
Suppose that $D = S \setminus (A \cup B \cup C) $. Observe that
$|D| = k - 3b < b$. Notice that the definitions of $A, B, C$ and $D$ depends on the coloring $c$ and the vertices $u_0, u_1, u_2$.

Now we  introduce some notations to aid our proof. 
Note that $c(u_i) \subset A \cup B \cup C \cup D$ 
for each $i \in \{0,1, \cdots , r-1 \}$. 
We will use the notation $c(u_i) = A_i B_i C_i D_i$ to denote the set of colors used on $u_i$. 
However, at some point of time if we are sure that some of the sets among 
$A, B, C$ or $D$ has empty intersection with $c(u_i)$, then we can drop the corresponding name of the set along with its subscript 
to denote $c(u_i)$. 
Moreover, if we are sure that $c(u_i) \cap X \neq \emptyset$
for some $X \in \{A,B,C,D\}$, then we can replace $X_i$ with $X^*_i$ in the above notation.

One more type of notation that we use is the following: if for some $X \in \{A,B,C,D\}$ we know that 
$c(u_i) \cap X  \neq \emptyset$, then we may denote it as 
$c(u_i)=X_{\#}$. 
For instance if $c(u_i)=A_{\#}$, then we are sure that $c(u_i)$ has colors from $A$ but it may or may not have colors from the other sets. 
Furthermore, if $c(u_i) \cap X  \neq \emptyset$ and 
$c(u_i) \cap Y  \neq \emptyset$ 
for some $X, Y \in \{A,B,C,D\}$, then we may denote it as $c(u_i)=X_{\#}Y_{\#}$. Similarly, we may use three or four set names among $A,B,C,D$ for this notation.

For instance, if at some point of time we learn that 
$c(u_i) \cap A \neq \emptyset$, $c(u_i) \cap B = \emptyset$ and
 $c(u_i) \cap C = \emptyset$, then we can write 
$c(u_i) = A_i B_i D_i$, $c(u_i) = A^*_i D_i$, 
$c(u_i) = A^*_i B^*_i D_i$,
$c(u_i) = A_{\#}B_{\#}$ or 
$c(u_i) = A_{\#}$ to denote the set of colors used on $u_i$. We will choose to use the notation that suits our purpose.

Now we are going to list some observations needed for the proof. 

\begin{observation}\label{obs p1}
There exists no index $i$ with $c(u_i) = D^*_i$. 
\end{observation}

\begin{proof}
It is not possible to have $c(u_i) = D^*_i$ as $|c(u_i)| = b$, 
$|D^*_i| \leq |D| < b$. 
\end{proof}

The next observation directly follows from the second condition from the definition of $b$-fold oriented $k$-coloring and the fact that  $c(u_0) = A$, $c(u_1) = B$, $c(u_2) = C$.

\begin{observation}\label{obs p2}
The following implications hold:
\begin{itemize}
\item[$(a)$]  $c(u_i) = B_{\#} \implies c(u_{i+1}) \neq A_{\#}$,

\item[$(b)$]  $c(u_{i+1}) = A_{\#} \implies c(u_i) \neq B_{\#}$,

\item[$(c)$]  if $c(u_3) = A$, then 
$c(u_i) = A_{\#} \implies c(u_{i+1}) \neq C_{\#}$,
\end{itemize}
 \end{observation}

In the above statement the case if $c(u_i) = A_{\#}$ is considered under the special additional condition $c(u_3) = A$.

\begin{observation}\label{obs p3}
If $c(u_i) = X_iD_i$ and $c(u_j) = X_jD_j$, then  
$X_iD_i \cap X_jD_j \neq \emptyset$ for any 
$(X_i,X_j) \in \{(A_i,A_j),(B_i,B_j),(C_i,C_j)\}$.
\end{observation}

\begin{proof}
This observation holds by the pigeonhole principle as 
$|c(u_i)| = |c(u_j)| = b$ and $|X \cup D| < 2b$.
\end{proof}

The next observation is based on the fact that the color sets 
assigned to the  three vertices belonging to a directed 2-path must be distinct under any $b$-fold oriented $k$-coloring as a directed 2-path is an oriented clique. 

\begin{observation}\label{obs p4}
For any index $i \in \{0,1, \cdots , r-1\}$ and for any 
$X \in \{A,B,C\}$ we have, 
$$X \cap (c(u_i) \cup c(u_{i+1}) \cup c(u_{i+2})) \neq \emptyset .$$
\end{observation}

\begin{proof}
Note that the vertices $u_i, u_{i+1}$ and $u_{i+2}$ of $C_r$ 
induces a directed 2-path, and hence $\{u_i, u_{i+1}, u_{i+2}\}$ is a an oriented relative clique. 
Thus  $|c(u_i) \cup c(u_{i+1}) \cup c(u_{i+2})| = 3b$.
However, $|S \setminus X| = k-b < 3b$. This implies
$X \cap (c(u_i) \cup c(u_{i+1}) \cup c(u_{i+2})) \neq \emptyset$. 
\end{proof}

The above observation shows that three consecutive vertices of $C_r$ cannot all avoid colors from a particular set $X \in \{A,B,C\}$. 
The following observation will show the opposite.

\begin{observation}\label{obs p5}
For any index $i \in \{0,1, \cdots , r-1\}$ and for any 
$X \in \{A,B,C\}$ we cannot have, 
$$c(u_i), c(u_{i+1}), c(u_{i+2}) = X_{\#} .$$
\end{observation}

\begin{proof}
Suppose the contrary. 

If $c(u_i), c(u_{i+1}), c(u_{i+2}) = A_{\#}$, then 
Observation~\ref{obs p2} implies 
$c(u_{i-1}), c(u_{i}), c(u_{i+1}) \neq B_{\#}$ contradicting 
Observation~\ref{obs p4}. The other two cases, namely, 
if $c(u_i), c(u_{i+1}), c(u_{i+2}) = B_{\#}$ and,
if $c(u_i), c(u_{i+1}), c(u_{i+2}) = C_{\#}$ can be proved similarly. 
\end{proof}

Based on the above observations we are able to show stronger implications.

\begin{observation}\label{obs p6}
The following implications hold:
\begin{itemize}
\item[$(a)$] $c(u_i) = C_\#$ and 
$c(u_{i}) \neq B_\# \implies   c(u_{i+1}) = A_\#$   and 
$c(u_{i+1}) \neq C_\#$,

\item[$(b)$] $c(u_i) = B_\#$ and 
$c(u_{i}) \neq A_\# \implies   c(u_{i+1}) = C_\#$   and 
$c(u_{i+1}) \neq B_\#$,

\item[$(c)$] $c(u_3) = A$, $c(u_i) = A_\#$ 
and $c(u_{i}) \neq C_\# \implies c(u_{i+1}) = B_\#$ and
$c(u_{i+1}) \neq A_\#$. 
\end{itemize}
 \end{observation}

\begin{proof}
$(a)$  If $c(u_i) = C_\#$ and 
$c(u_{i}) \neq B_\#$, then $c(u_{i+1}) \neq B_\#$ due to Observation~\ref{obs p2}.
Moreover, if   $c(u_{i+1}) = C_\#$, then $c(u_{i+2}) \neq B_\#$
 due to Observation~\ref{obs p2}.
This implies $c(u_i),c(u_{i+1}),c(u_{i+2}) \neq B_\#$ contradicting 
Observation~\ref{obs p4}. 
Also $c(u_{i+1}) = D^*_{i+1}$ is not possible due to Observation~\ref{obs p1}.
Thus $c(u_{i+1}) = A_\#$ and $c(u_{i+1}) \neq C_\#$.

\medskip

$(b), (c)$ The proofs are similar. 
\end{proof}

An \textit{automorphism} of an oriented graph $G$ is a function 
$f : V(G) \rightarrow V(G)$ such that 
$f(u)f(v) \in A(G)$ if and only if $uv \in A(G)$.
An oriented graph $G$ is \textit{$3$-dipath transitive} if for any two 
directed $3$-paths $abcd$ and $wxyz$ of $G$ there exists an automorphism $f$ 
such that $f(a) =w$, $f(b) =x$, $f(c) =y$ and $f(d) =z$.

Now we are going to prove one of the key lemmas to aid the proof of the theorem. We need to introduce some notations to state the lemma. 
The set of colors $c(u_i)$ is called the \textit{label} of $u_i$ for 
any $i \in \mathbb{Z}/r\mathbb{Z}$. 
A sequence of \textit{$k$ consecutive labels} are 
the labels used on
$u_{i}, u_{i+1}, \cdots, u_{i+k-1}$ for some 
$i \in \mathbb{Z}/r \mathbb{Z}$. 
A sequence of $3$ consecutive labels  of the form 
($A^*_iD_i, A_{i+1}B^*_{i+1}D_{i+1},  C^*_{i+2}D_{i+2}$) 
is  a \textit{triple}. 
The sequence of $4$ consecutive labels of the form 
($A^*_iD_i, A^*_{i+1}B^*_{i+1}D_{i+1},
B^*_{i+2}C^*_{i+2}D_{i+2},  C^*_{i+3}D_{i+3}$) is a \textit{quad}.

\begin{lemma}\label{lem triple-quad} 
Any label is either part of a triple or a quad. 
\end{lemma}

\begin{proof}
As $c(u_0) = A$,  $c(u_1) = B$, $c(u_2) = C$, the lemma is true for all $c(u_i)$ having $ i \in \{0,1,2 \}$.

Now assume that the lemma is true for all $c(u_i)$ having 
$ i \in \{0,1, \cdots k-1 \}$.
Observe that $c(u_{k-1}) = C^*_kD_k$ due to our assumption. 
Now we want to show that  $c(u_k)$ is part of  either a 
triple or a quad.

Note that $c(u_k) = A^*_kD_k$ due to 
Observation~\ref{obs p2} and~\ref{obs p6}.

If $c(u_{k+1}) \neq B_{\#}$, then it contradicts 
Observation~\ref{obs p4}. 
Thus $c(u_{k+1}) = A_{i+1}B^*_{i+1}C_{i+1}D_{i+1}$. 
However if $c(u_{k+1}) = C_{\#}$, 
then $c(u_2) \cap c(u_{k+1}) \neq \emptyset $.
On the other hand, as $c(u_3) = A^*_{3}D_{3}$, 
we have $c(u_3) \cap c(u_{k}) \neq \emptyset $.
This contradicts the second condition of the definition  
of $b$-fold oriented $k$-coloring. 
Therefore, 
$c(u_{k+1}) = A_{i+1}B^*_{i+1}D_{i+1}$.

Also  $c(u_{k+1}) = B_{\#}$ implies 
$c(u_{k+2}) \neq A_{\#}$ due to Observation~\ref{obs p2}.
Moreover,  $c(u_{k+2}) = C_{\#}$ due to Observation~\ref{obs p4} 
Therefore, 
either $c(u_{k+2}) = C^*_{i+2}D_{i+2}$ making 
$c(u_k)$ part of a triple, or 
$c(u_{k+2}) = B^*_{i+2}C^*_{i+2}D_{i+2}$. 
In the later case, 
we must have $c(u_{k+3}) = C^*_{i+3}D_{i+3}$ due to 
Observation~\ref{obs p2}. 
Also this implies that $c(u_{k+1}) = A_{\#}$ due to 
Observation~\ref{obs p4}. This makes $c(u_k)$ part of a quad. 
\end{proof}

The labels in the entire cycle can thus be viewed 
as a circular arrangement of triples and quads.  

Let $(c(u_t), c(u_{t+1}), c(u_{t+2}))$ be a triple.
If we rename  the vertices as $u_i = u_{i+t}$ for all 
   $i \in \mathbb{Z}/r \mathbb{Z}$ and some constant $t$, then 
the definitions of the sets $A, B, C$ and $D$ gets  changed accordingly. 
However, we observe that a triple remains a triple and a quad remains a quad.

\begin{lemma}\label{lem same-same}
If we rename the vertices as $u_i = u_{i+t}$ for all 
$i \in \mathbb{Z}/r \mathbb{Z}$ and some constant $t$
where $\Sigma = (c(u_t), c(u_{t+1}), c(u_{t+2}))$ is a triple, then
a triple remains a triple and a quad remains a quad. 
\end{lemma}

\begin{proof}
For convenience we will refer to every vertex as per their names before renaming throughout the proof.

After renaming, the definitions of the sets $A, B, C$ and $D$ changes. We will 
use the labels $c(u_{t}) = A', c(u_{t+1}) =B', c(u_{t+2})=C'$ 
and $D'= (A \cup B \cup C \cup D) \setminus  (A' \cup B' \cup C')$
for convenience.

Suppose the first triple/quad after $\Sigma$ that does not remain a triple/quad 
after renaming is $\Lambda $. We consider the following two cases.

\medskip

\noindent \textbf{Case 1:}  Suppose before renaming 
 $\Lambda = (c(u_{j}), c(u_{j+1}), c(u_{j+2}))$ was a triple.

Considering the situation before renaming, due to pigeonhole principle, we can say that 
$c(u_{t}) \cap c(u_{j+3}) = A^*_tD_t \cap A^*_{j+3}D_{j+3} \neq \emptyset $. 

On the other hand, considering the situation after renaming
$c(u_{t}) \cap c(u_{j+3}) = A' \cap C'^*_{j+3}D'_{j+3} = \emptyset $, a contradiction. Thus $\Lambda$ must be a triple even after renaming. 

\medskip

\noindent \textbf{Case 2:}  Suppose before renaming  $\Lambda = (c(u_{j}), c(u_{j+1}), c(u_{j+2}), c(u_{j+3}))$ 
  was a quad.

Considering the situation before renaming, due to pigeonhole principle, we can say that 
$c(u_{t+2}) \cap c(u_{j+3}) = C^*_{t+2}D_{t+2} \cap C^*_{j+3}D_{j+3} \neq \emptyset $.

On the other hand, considering the situation after renaming
$c(u_{t+2}) \cap c(u_{j+3}) 
= C' \cap A'^*_{j+3}D'_{j+3} = \emptyset $, a contradiction. 
Thus $\Lambda$ must be a quad even after renaming. 
\end{proof}

If the last element of a triple/quad $\Sigma$ is $c(u_i)$ and the first element of a triple/quad $ \Lambda $ is $c(u_{i+1})$, then 
$ \Sigma $ and $ \Lambda $ are \textit{consecutive} where 
$ \Sigma $ is \textit{before} $ \Lambda $ and  
$ \Lambda $  is \textit{after} $ \Sigma $.

\begin{lemma}\label{no-consecutive-triples}
There are no consecutive triples. 
\end{lemma}

\begin{proof}
As $r \not\equiv 0 \bmod 3$, there must be at least one quad. 
Thus if there are  consecutive triples, then there 
exists a quad $\Sigma_1$ and two triples $\Sigma_2, \Sigma_3$ 
such that $\Sigma_1, \Sigma_3$  are, respectively, before and after $\Sigma_2$.

As the sets $A, B$  and $C$ are defined based on the set of colors assigned to the vertices $u_0, u_1$ and $u_2$, respectively, 
and as $C_r$ is $3$-dipath transitive, we may assume that 
$\Sigma_2 = (c(u_0), c(u_1), c(u_2)) = (A,B,C)$. 
Therefore, 
$\Sigma_1 = (A^*_{-4}D_{-4}, A^*_{-3}B^*_{-3}D_{-3},
B^*_{-2}C^*_{-2}D_{-2},  C^*_{-1}D_{-1})$
and 
$\Sigma_3 = (A^*_{3}D_{3},A_4B^*_{4}D_{4},C^*_{5}D_{5})$.

 Note
that the following constraints follow directly from the 
definition of $b$-fold oriented coloring. 
For any distinct $i,j \in \mathbb{Z}/r \mathbb{Z}$

\begin{equation}\label{2dipath-equation}
c(u_i) \cap c(u_j) \neq \emptyset \implies 
c(u_{i-1}) \cap c(u_{j+1}) = \emptyset \text{ and } 
c(u_{i+1}) \cap c(u_{j-1}) = \emptyset.
\end{equation}

Observe that 
$c(u_{-1}) \cap c(u_5) = C_{-1} D_{-1} \cap C_5 D_5 \neq \emptyset $
by pigeonhole principle. This implies
$ c(u_{0}) \cap c(u_4) = \emptyset $ by 
$ eq^n$~(\ref{2dipath-equation}).

Furthermore $c(u_0) \cap c(u_3) = A \cap A^*_3 D_3 \neq \emptyset $.
This implies
$ c(u_{-1}) \cap c(u_4) = \emptyset $ by 
$ eq^n$~(\ref{2dipath-equation}).

By $ eq^n$~(\ref{2dipath-equation}) we have 
$$c(u_{-1}) \cap c(u_2) = C^*_{-1}D_{-1} \cap C \neq \emptyset 
\implies  c(u_{-2}) \cap c(u_3) = \emptyset $$ and
$$c(u_{-2}) \cap c(u_2) = B^*_{-2}C^*_{-2}D_{-2} \cap C \neq \emptyset 
\implies  c(u_{-3}) \cap c(u_3) = \emptyset .$$
As $c(u_{-2}) \cap c(u_3) = c(u_{-3}) \cap c(u_3) = \emptyset $, Lemma~\ref{lem dipath} implies 
$ c(u_{-1}) \cap c(u_3) \neq \emptyset $.
This implies $ c(u_{-2}) \cap c(u_4) = \emptyset $.

Thus  $c(u_{-2}) \cup c(u_{-1}) \cup c(u_{0}) \subseteq 
(A \cup B \cup C \cup D) \setminus c(u_4)$. This is a contradiction to Lemma~\ref{lem dipath} as     
$|(A \cup B \cup C \cup D) \setminus c(u_4)| < 3b $.
\end{proof}

Therefore, there are some positive number of quads in between any two triples.  
In order to analyze the number of quads in between two triples,  we construct
a symmetric binary matrix $T$ of dimension $r \times r$ such that 

\begin{equation}\label{T-matrix}
T[i,j]=1 \text{ if and only if } c(u_i) \cap c(u_j) \neq\emptyset
\text{ for all } i,j \in \mathbb{Z}/r \mathbb{Z}.
\end{equation}

Thus $T[i,j]=1$  implies $T[i+1,j-1]=T[i-1,j+1]=0$
by $eq^n$~(\ref{2dipath-equation}).
The $i^{th}$ row of the binary matrix $T$ is denoted by $T_i$.
Three consecutive bits along a row or a column cannot be  all $0$s or all $1$s due to Observations~\ref{obs p4} and~\ref{obs p5}.

A \textit{cyclic right-shift} is the operation of rearranging the entries in a vector  by moving the final entry to the first position, while shifting all other entries to the next position.

Now we will present an important property of the binary matrix $T$.

\begin{lemma}\label{cyclic-shift}
The vector $T_{i+1}$ is obtained by a cyclic right-shift on  $T_i$.
\end{lemma}

\begin{proof}
We have to show that $T[i,j]=T[i+1,j+1]$ 
for any $i,j \in \mathbb{Z}/r \mathbb{Z}$. 
Suppose the contrary, that is, 
$T[i,j] \neq T[i+1,j+1]$ for some
$i,j \in \mathbb{Z}/r \mathbb{Z}$.

First assume that $T[i,j] = 0$ and $T[i+1,j+1]= 1$. This implies 
 $T[i+2,j]=T[i,j+2]=0$ by $eq^n$~(\ref{2dipath-equation}).
 Moreover, to avoid three consecutive $0$s in the $i^{th}$ row
 we have  $T[i,j+1]=1$. This implies $T[i+1,j]=0$ by $eq^n$~(\ref{2dipath-equation}). 
 Now the $j^{th}$ column 
 has three consecutive $0$s contradicting Observation~\ref{obs p4}. 
 Thus $T[i,j] = 0$ and $T[i+1,j+1]= 1$ is not possible.

 Similarly it can be shown that it is not possible to have 
 $T[i,j]=1$ and $T[i+1,j+1]=0$.
\end{proof}

Let $\Sigma_1$ be a triple and after that there are $q$ quads after which there is a triple $\Sigma_2$. 
Then we say that $\Sigma_2$ is \textit{$q$-quads after} $\Sigma_1$.

\begin{lemma}\label{constant-triple-separation}
If $\Sigma_3$ is $q_2$-quads after $\Sigma_2$ and 
$\Sigma_2$ is $q_1$-quads after $\Sigma_1$, then $q_2 = q_1$. 
\end{lemma}

\begin{proof}
Without loss of generality let 
$\Sigma_1=(u_{-i-2},u_{-i-1},u_{-i})$, 
$\Sigma_2=(u_0,u_1,u_2)$, $\Sigma_3=(u_j,u_{j+1},u_{j+2})$
and $q_1 < q_2$.

As  $c(u_{-i-1}) \neq C_{\#}$, we have
$c(u_{-i-1}) \cap c(u_2) = \emptyset $ 
and thus due to Lemma~\ref{cyclic-shift}
 $c(u_{2}) \cap c(u_{i+5}) =  \emptyset $.
 However, note that as $q_1 < q_2$, $c(u_{i+5})$ is the fourth element of a quad. 
 Thus  $c(u_{i+5}) = C^*_{i+5}D_{i+5}$ implying 
$c(u_{2}) \cap c(u_{i+5}) \neq  \emptyset $, a contradiction.
\end{proof}

Note that there is at least one triple due to the way the sets
 $A, B, C$ are defined.
If there are $t$ triples and if a triple is $q$-quads after another, then $r=(4q+3)t$.
Therefore, the next lemma follows directly.

\begin{lemma}\label{lem typeB none}
There is no  miser $b$-fold oriented $k$-coloring of $C_r$ unless $r$ is divisible by $(4q+3)$ for some $q \in \mathbb{N}$.
\end{lemma}

Suppose that $r$ is divisible by $(4q+3)$ for some 
$q \in \mathbb{N}$. 
Furthermore, let $q'$ is the  smallest integer
such that $(4q'+3)$   divides $r$. Then 
observe that $(4q'+3)$ is a prime number. Therefore,
if a $C_r$ admits a miser $b$-fold oriented $k$-coloring then $r$ must have a 
type-$A$ prime factor. 

Furthermore 
we will show that if $r$ is divisible by a type-$A$ prime $p$, then  
$\chi^*_o(C_r) \leq 4 - \frac{4}{p+1}$. In fact we will provide a
particular coloring to prove that.

\begin{lemma}\label{lem typeA upper} 
If a type-$A$ prime $p$ divides $r$, 
then $\chi^*_o(C_r) \leq 4 - \beta(r)$.
\end{lemma}

\begin{proof}
Assume that $p$ is the smallest type-$A$ prime that divides $r$.
Let $p = 4k-1$ and $r = pt$. Now we will provide a 
$k$-fold oriented $p$-coloring of $C_r$.

Let $f(x) = \{i \cdot k +1, i \cdot k +2, \cdots, i \cdot k +k \}$
where `+' is taken modulo $(4k-1)$,  $x \equiv i \bmod p$ and $i \leq p-1$ for all $x \in V(C_r)$. 
Observe  that 
$f$ is a $k$-fold oriented $p$-coloring of $C_r$. 
\end{proof}

Also we will show that the above upper bound is the best possible. 

\begin{lemma}\label{lem typeA lower} 
If a type-$A$ prime $p$ divides $r$, 
then $\chi^*_o(C_r) \geq 4 - \beta(r)$.
\end{lemma}

\begin{proof}
Suppose a miser $b$-fold oriented $k$-coloring $c$ of $C_r$ has exactly $t$ triples
and $q$ quads in between two successive triples.
Therefore, $r = (4q+3)t$ and $\beta(r) \leq q+1$.

Note that a color from the set $A$ can be used at most once in a triple or a quad. Thus, a particular color from $A$ is used at most 
$(q+1)t$ times. 
As we can rename the vertices as $u_i = u_{i+t}$, any color used 
can be thought of as a color from $A$. Hence any color is used at most  $(q+1)t$ times. 

Thus $$rb \leq k(q+1)t 
\implies 4 - \frac{1}{(q+1)} \leq \frac{k}{b}
\implies 4 - \beta(r) \leq \frac{k}{b}.$$
This ends our proof.
\end{proof}

Now we are ready to prove Theorem~\ref{directed-cycles}(d).

\medskip

\noindent  \textit{Proof of Theorem~\ref{directed-cycles}(d).}
The proof directly follows from 
Lemma~\ref{lem typeB none}, \ref{lem typeA upper} and~\ref{lem typeA lower}. \hfill $ \square $

\section{Sparse planar graphs}\label{sec planar}
Recall that $\mathcal{P}_g$ denote the family of oriented planar graphs having girth at least $g$. 
Due to Theorem~\ref{directed-cycles} proved in the previous section, we know that 
$\chi^*_o(\mathcal{P}_{g}) \geq 4$ for all $g \geq 3$. A natural question to ask is, how close to $4$ can  $\chi^*_o(\mathcal{P}_{g})$ get as we increase the value of $g$. The next result is our attempt to answer that question.

\begin{theorem}\label{th epsilon-girth}
For any $\epsilon > 0$, there exists an integer $g_{\epsilon} < \infty$  
such that 
$4 \leq \chi^*_o(\mathcal{P}_{g_{\epsilon}}) \leq 4+ \epsilon$. 
\end{theorem}

We need to introduce a few notation and terminology for our proof. 
Given an arc $xy$ of an oriented graph $G$, $y$ is an \textit{out-neighbor} of $x$ and 
$x$ is an \textit{in-neighbor} of $y$. 
The set of all out-neighbors (resp., in-neighbors) of $x$ is denoted by 
$N^+(x)$ (resp., $N^-(x)$). Furthermore, given a vertex set $S \subseteq V(G)$ one can define the following
$$N^+(S) = \cup_{x \in S} N^+(x) \text{ and } N^-(S) = \cup_{x \in S} N^-(x).$$

Let $\alpha = (\alpha_1, \alpha_2, \cdots, \alpha_k) \in \{+,-\}^k$ 
be an $k$-vector. 
For $k=1$,  let us define $N^{\alpha}(x)=N^{\alpha_1}(x)$ for any $x \in V(G)$.
 After that, for $k \geq 2$, we can inductively 
define $N^{\alpha}(x) = N^{\alpha_1}(N^{\alpha'}(x))$ where 
$\alpha' = (\alpha_2, \alpha_3, \cdots, \alpha_k)$. 
An oriented graph $G$ is $k$-nice if for each $k$-vector $\alpha \in \{+,-\}^k$ and each vertex $x \in V(G)$ we have $N^{\alpha}(x) = V(G)$.
This is an important property in our context due to the following result.

\begin{proposition}\cite{nice}\label{prop nice}
If an oriented graph $T$ is $k$-nice, then any $G \in \mathcal{P}_{5k-1}$ admits a homomorphism to $T$. 
\end{proposition}

Next we need to construct a particular class of oriented graph
$T_l$ for all $l \geq 0$. Let $m = 2^l$ and let $n = 4m+1 = 2^{l+2} +1$. 
Consider the additive group 
$\mathbb{Z}/n\mathbb{Z}$ and let $V(T_l)$ be the set of all 
$m$-tuples of $\mathbb{Z}/n\mathbb{Z}$ having consecutive elements, that is, 
$$V(T_l) = \{x_i = (i, i+ 1, \cdots, i+m-1) :
i \in  \mathbb{Z}/n\mathbb{Z} \}$$ where the $+$ operation is 
taken modulo $n$.
Furthermore, $$A(T_l) = \{x_ix_j | i+m  \leq j  \leq i + m+1 \}.$$

Intuitively speaking, the vertex $x_i$ is represented by a set of $m$-tuples 
 $(i, i+ 1, \cdots, i+m-1)$ and it has  arcs to the vertices $x_{i+m}= (i+m, i+m +1, \cdots, i+2m-1)$ and $x_{i+m+1} = (i+m+1, i+m +2, \cdots, i+2m)$. Note that the elements of the tuple of $x_i$ are distinct from the elements of the tuple of $x_{i+m}$ and $x_{i+m+1}$. 
Thus $x_i$ has exactly two out-neighbors $x_{i+m}$ and $x_{i+m+1}$.
On the other hand, $x_{i-m}$ and $x_{i-m-1}$ has $x_i$ as their common out-neighbor. 
That is, $x_i$ has exactly two in-neighbors  $x_{i-m}$ and $x_{i-m-1}$.  
 For convenience we have depicted the  oriented graph $T_l$ for $l = 1$ in Fig.~\ref{figure:sample}.

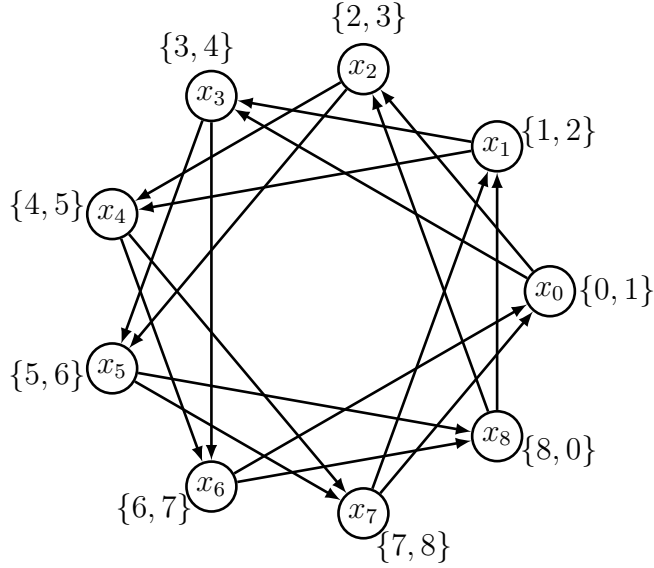
\begin{figure}
\centering
	\begin{tikzpicture}[inner sep=.7mm]
	
	\foreach \a in {0,...,8} 
	{
		\node[draw, circle, line width=1pt](v\a) at (\a*360/9:3cm){$x_\a$};	
	}

	\node at (0:3.3cm)[right]{$\{0,1\}$};	
	\node at (360/9:3.3cm)[right]{$\{1,2\}$};	
	\node at (2*360/9:3.4cm)[above]{$\{2,3\}$};	
	\node at (3*360/9:3.4cm)[above]{$\{3,4\}$};
	\node at (4*360/9:3.3cm)[left]{$\{4,5\}$};
	\node at (5*360/9:3.3cm)[left]{$\{5,6\}$};
	\node at (6*360/9:3.35cm)[left]{$\{6,7\}$};
	\node at (7*360/9:3.45cm)[right]{$\{7,8\}$};
	\node at (8*360/9:3.3cm)[right]{$\{8,0\}$};

	\draw[-latex,line width=1pt,black] (v0) -- (v3);
	\draw[-latex,line width=1pt,black] (v1) -- (v4);
	\draw[-latex,line width=1pt,black] (v2) -- (v5);
	\draw[-latex,line width=1pt,black] (v3) -- (v6);
	\draw[-latex,line width=1pt,black] (v4) -- (v7);
	\draw[-latex,line width=1pt,black] (v5) -- (v8);
	\draw[-latex,line width=1pt,black] (v6) -- (v0);
	\draw[-latex,line width=1pt,black] (v7) -- (v1);
	\draw[-latex,line width=1pt,black] (v8) -- (v2);

	\draw[-latex,line width=1pt,black] (v0) -- (v2);
	\draw[-latex,line width=1pt,black] (v1) -- (v3);
	\draw[-latex,line width=1pt,black] (v2) -- (v4);
	\draw[-latex,line width=1pt,black] (v3) -- (v5);
	\draw[-latex,line width=1pt,black] (v4) -- (v6);
	\draw[-latex,line width=1pt,black] (v5) -- (v7);
	\draw[-latex,line width=1pt,black] (v6) -- (v8);
	\draw[-latex,line width=1pt,black] (v7) -- (v0);
	\draw[-latex,line width=1pt,black] (v8) -- (v1);

 	\end{tikzpicture}
 	\caption{The oriented graph $T_l$ for $l = 1$.}
 	\label{figure:sample}
 \end{figure}

Next we will show that $T_l$ is $n$-nice.

\begin{lemma}
The oriented graph $T_l$ is $n$-nice. 
\end{lemma}

\begin{proof}
Observe that each vertex $x_i$ of $T_l$ has exactly $2$ out-neighbors $x_{i+m} , x_{i+m+1}$ and 
exactly $2$ in-neighbors $x_{i-m}, x_{i-m-1}$. Therefore, 
given a set $S = \{x_i, x_{i+1}, \cdots, x_{i+m-1}\}$  of $t$ consecutive neighbors of 
$T_l$, we have $|N^+(S)|=|N^-(S)|=t+1$ for all $t < n$. Thus, for any vertex $x_i$ and any 
$(n-1)$-vector $\alpha$, we will have $|N^{\alpha}(x_i)| = n$. Hence $N^{\alpha}(x_i) = V(T_l)$. 
\end{proof}

Therefore using the above lemma and Proposition~\ref{prop nice}  we can conclude 
the following.

\begin{lemma}\label{lem girth-nice}
Any $G \in \mathcal{P}_{5 \cdot n -1}$ admits a homomorphism to $T_l$ where $n = 2^{l+2} + 1$.
\end{lemma}

Now we are ready to prove Theorem~\ref{th epsilon-girth}. 

\medskip

\noindent  \textit{Proof of Theorem~\ref{th epsilon-girth}}
Note that a directed cycle $C_r$ of length $r = 5 \cdot 2^{n+1}$ belongs to 
$\mathcal{P}_{5 \cdot n -1}$. Also, $\beta(r) = 0$. Thus according to Theorem~\ref{directed-cycles}(d), 
$\chi_o^*(C_r) = 4$. This implies the lower bound. 
We have independently proved Theorem~\ref{directed-cycles}(d) in the next section.

Next we are going to prove the upper bound. 
Note that each vertex of $T_l$ corresponds to an $m$-tuple. where $m = 2^{l}$. 
That means if we consider the elements of the tuples to be colors, then the tuple corresponding to $x_i$ can be viewed as set of $m$-colors assigned to $x_i$. Therefore, the tuples corresponding to the vertices of $T_l$ is a potential $m$-fold oriented $n$-coloring of $T_l$.

We are going to verify if the tuples indeed provide an $m$-fold oriented $n$-coloring of $T_l$. 
To verify the first condition of $m$-fold oriented coloring, we observe that 
any two adjacent vertices of $T_l$ corresponds to tuples that have disjoint elements. To be precise, 
an arc of $T_l$ can be of the two following types: $x_ix_{i+m}$ and $ x_ix_{i+m+1}$. 
Clearly, the elements of $x_i = (i, i+1, \cdots, i+m-1)$ is different from the elements of 
$x_{i+m} = (i+m, i+m+1, \cdots, i+2m-1)$ and $x_{i+m+1} = (i+m+1, i+m+2, \cdots, i+2m)$.

To verify the second condition, let us assume that the condition is violated by the set of arcs 
$x_{i_1}x_{i_2}$ and $x_{j_1}x_{j_2}$. In order for such a violation to happen, we must have 
at least one common element $s$
between the tuples corrsponding to $x_{i_1}, x_{j_2}$ and
at least one common element $t$
between the tuples corrsponding to $x_{i_2}, x_{j_1}$.

That $s$ is an element of $x_{i_1}$ and $t$ is an element of $x_{i_2}$ implies 
$t = s+l$ where $1 \leq l \leq 2m$. 
Similarly, as $t$ is an element of $x_{j_1}$ and $s$ is an element of $x_{j_2}$, we have  
$s = t+l'$ where $1 \leq l' \leq 2m$. Therefore, $s + (l+l') = s$ where $2 \leq (l+l') \leq 4m$. This is a contradiction as we are working 
 modulo $(4m+1)$. 

Hence, $T_l$ admits an $m$-fold oriented $n$-coloring. Therefore, any graph that admits a homomorphism to $T_l$ also admits an $m$-fold oriented $n$-coloring.  
Thus if $G$ admits a homomorphism to $T_l$, then 
$$\chi^*_o(G) \leq \frac{|V(T_l)|}{m} = \frac{n}{m} =\frac{2^{l+2}+1}{2^{l}}=4+\frac{1}{2^{l}}.$$ Thus we are done as 
$\lim_{l \to \infty}  \frac{1}{2^{l}} =0$. 
 \hfill $ \square $

\section{Conclusions}\label{sec conclusions}
This article initiates the study on fractional  oriented  chromatic number. The works of this article naturally motivates the following problems. 
\begin{enumerate}[(1)]
\item Find the linear programming formulation of the notion of fractional  oriented  chromatic number. 

\item Find out the fractional  oriented chromatic numbers of all oriented cycles. 

\item Find out the fractional  oriented chromatic numbers of all the families of oriented planar graphs having girth at least $g$, for all $g \geq 3$.

\item Study the complexity dichotomy of finding the parameter.
\end{enumerate}

\bibliographystyle{abbrv}
\bibliography{TOCReferences,NSS14}

\begin{thebibliography}{1}

\bibitem{courcelle1994monadic}
B.~Courcelle.
\newblock The monadic second order logic of graphs vi: On several
  representations of graphs by relational structures.
\newblock {\em Discrete Applied Mathematics}, 54(2):117--149, 1994.

\bibitem{nice}
P.~Hell, A.~V. Kostochka, A.~Raspaud, and {\'E}.~Sopena.
\newblock On nice graphs.
\newblock {\em Discrete Mathematics}, 234(1--3):39--51, 2001.

\bibitem{nandySS2016}
A.~Nandy, S.~Sen, and {\'{E}}.~Sopena.
\newblock Outerplanar and planar oriented cliques.
\newblock {\em Journal of Graph Theory}, 82(2):165--193, 2016.

\bibitem{raspaud1994planar80}
A.~Raspaud and {\'E}.~Sopena.
\newblock Good and semi-strong colorings of oriented planar graphs.
\newblock {\em Information Processing Letters}, 51(4):171--174, 1994.

\bibitem{scheinermann2008fgt}
E.~R. Scheinermann and D.~H. Ullman.
\newblock {\em Fractional Graph Theory: A Rational Approach to the Theory of
  Graphs}.
\newblock 2008.

\bibitem{sopena-updated-survey}
{\'E}.~Sopena.
\newblock Homomorphisms and colourings of oriented graphs: An updated survey.
\newblock {\em Discrete Mathematics}, 339(7):1993--2005, 2016.

\end{thebibliography}

\end{document}